\documentclass{amsart}
\usepackage{amsthm,amsfonts,amsmath,amscd,amssymb,latexsym,epsfig}

\newcommand{\ssl}{{\mathfrak s \mathfrak l}}

\newcommand{\g}{{\mathfrak g}}         

\newcommand{\cx}{{\mathbb C}}
\newcommand{\diag}{\operatorname{diag}}

\newcommand{\tr}{\operatorname{tr}}

\newcommand{\im}{\operatorname{Im}}

\newcommand{\End}{\operatorname{End}}
\newcommand{\sEnd}{{\mathcal End}}

\newcommand{\Ker}{\operatorname{Ker}}

\newcommand{\Hilb}{\operatorname{Hilb}}

\newcommand{\Ext}{\operatorname{Ext}}

\newcommand{\Mat}{\operatorname{Mat}}

\newcommand{\Spec}{\operatorname{Spec}}

\newcommand{\ol}{\overline}

\numberwithin{equation}{section}

\newtheorem{theorem}{Theorem}[section]

\newtheorem{lemma}[theorem]{Lemma}

\newtheorem{corollary}[theorem]{Corollary}

\newtheorem{proposition}[theorem]{Proposition}

\theoremstyle{remark}

\newtheorem{remark}[theorem]{Remark}

\newtheorem{definition}[theorem]{Definition}

\newtheorem{example}[theorem]{Example}

\newcommand{\oC}{{\mathbb{C}}}

\newcommand{\oN}{{\mathbb{N}}}

\newcommand{\oP}{{\mathbb{P}}}

\newcommand{\oR}{{\mathbb{R}}}

\newcommand{\oZ}{{\mathbb{Z}}}

\newcommand{\sA}{{\mathcal{A}}}   

\newcommand{\sD}{{\mathcal{D}}}

\newcommand{\sF}{{\mathcal{F}}}

\newcommand{\sI}{{\mathcal{I}}}

\newcommand{\sK}{{\mathcal{K}}}
\newcommand{\sL}{{\mathcal{L}}}   
\newcommand{\sM}{{\mathcal{M}}}   
\newcommand{\sN}{{\mathcal{N}}}
\newcommand{\sO}{{\mathcal{O}}}
\newcommand{\sP}{{\mathcal{P}}}

\newcommand{\sT}{{\mathcal{T}}}

\newcommand{\sZ}{{\mathcal{Z}}}

\newcommand{\fG}{{\mathfrak{g}}}
\newcommand{\fH}{{\mathfrak{h}}}

\newcommand{\fL}{{\mathfrak{l}}}

\begin{document}

\title{Hilbert schemes, commuting matrices, and hyperk\"ahler geometry}
\author{Roger Bielawski \and Carolin Peternell}
\address{Institut f\"ur Differentialgeometrie,
Leibniz Universit\"at Hannover,
Welfengarten 1, 30167 Hannover, Germany}
\subjclass[2010]{14C05, 14H50, 53C26, 53C28}



\begin{abstract} We represent algebraic curves via commuting matrix polynomials. This allows us to show that the Hilbert scheme of cohomologically stable nonplanar curves of genus $0$ and degree $d$ in $\oP^3\backslash\oP^1$ is isomorphic to a complexified hyperk\"ahler quotient of an open subset of a vector space by  a nonreductive Lie group. 
\end{abstract}

\maketitle
\thispagestyle{empty}


It has been observed in \cite{sigma} that the Hilbert scheme of real cohomologically stable (i.e.\ satisfying $h^0(\sN(-2))=0$) nonplanar  curves of fixed genus and degree in $\oP^3$, not intersecting a fixed real line, carries a natural pseudo-hyperk\"ahler structure. Here ``real" means invariant under a fixed point free antilinear involution of $\oP^3$. In the case of $g=0$ and $d=3$ this pseudo-hyperk\"ahler structure was shown in \cite{sigma} to be flat, and in fact, the manifold of cohomologically stable pure-dimensional Cohen-Macaulay nonplanar  curves in $\oP^3\backslash \oP^1$ with Hilbert polynomial $3n+1$ with its natural complexified hyperk\"ahler structure was shown there to be isomorphic to $\cx^{12}\simeq \cx^3\otimes \Mat_{2\times 2}(\cx)$. The proof of this relies on such curves being ACM and so clearly different methods are needed to study the pseudo-hyperk\"ahler geometry of the corresponding open subset of $\Hilb_{d,g}$ for other values of $d$ and $g$. 
\par
In the present article we present such a method via a correspondence between algebraic curves equipped with a flat projection onto $\oP^1$ and commuting matrix polynomials.
This correspondence allows us to describe the locus of cohomologically stable nonplanar curves of arithmetic genus $0$ and degree $d$ in $\oP^3\backslash \oP^1$ as a complexified hyperk\"ahler quotient of (an open subset of) a vector space by a {\em nonreductive} Lie group (Theorem \ref{main}). Formally, our moment map equations are very similar to the complex ADHM equations used by   Frenkel and Jardim \cite{FJ} in their construction of admissible torsion-free sheaves on $\oP^3$. The main difference is that the Lie group acting on solutions is no longer reductive.
\par
Restricting this description to real curves, we obtain the above pseudo-hyperk\"ahler structure for $g=0$ and any odd $d$ (for even $d$ there are no real rational curves in the above sense) as a hyperk\"ahler quotient of an open subset of a  flat quaternionic vector space by a nonreductive Lie group. There is also an analogous description of the natural hypersymplectic structure on cohomologically stable nonplanar curves of genus $0$ and any degree which are invariant under an antilinear involution of $\oP^3$, the fixed point set of which is $\oR\oP^3$. 
\par
We briefly describe the structure and the content of the paper. In the next section we provide alternative (to the one given by Nakajima \cite{Nak}) descriptions of the Hilbert scheme $(\cx^2)^{[n]}$ and of its open subset of non-collinear points. In \S 2 we discuss the above-mentioned correspondence between algebraic curves and commuting matrix polynomials. In \S 3 we restrict our attention to space curves, and show how to construct a twistor space $Z_d\to \oP^1$, the sections of which can be identified with either a) algebraic curves of fixed degree, or b) equivalence classes of certain commuting pairs of polynomials. Section 4 discusses complexified hyperk\"ahler and hypercomplex structures and their quotients. Finally, in \S 5 we apply these ideas and results to  genus $0$ space curves and show that in this case the description b) of sections of $Z_d$ is equivalent to a complexified hyperk\"ahler quotient. 

\begin{remark} Throughout the paper a {\em curve} means a projective Cohen-Macaulay scheme of pure dimension $1$.\label{curve}\end{remark}
                                                                                                                                                                                                                                                                \vspace{1mm}
                                                                                                                                                                                                                                                                
{\em Acknowledgement.}  This work has been carried out while both authors were members of, and the second author was fully funded by the DFG Priority Programme 2026 ``Geometry at infinity", the support of which is gratefully acknowledged. The authors also thank Michael Bulois and Israel Vainsencher for pointing out mistakes in an earlier version of the paper. Last, but not least, we are grateful to the anonymous referees for many helpful comments, which resulted in a greatly improved presentation.

\section{Hilbert schemes of points in $\cx^2$ and commuting matrices\label{comm}}

We begin with the following easy observation (cf. \cite{Nak} if $k=2$ and \cite[Thm. 2.5]{HJ} for general $k$):
\begin{proposition} There exists a natural set-theoretic bijection between the Hilbert scheme $(\cx^k)^{[n]}$ of $n$ points in $\cx^k$ and $GL(n,\cx)$-orbits of $k$-tuples $(A_1,\dots,A_k)$ of $n\times n$ matrices such that $\cx[A_1,\dots,A_k]$ is an $n$-dimensional commutative algebra which is conjugate to its image in $\End(\cx^n)$ under the regular representation.\label{bijection}
\end{proposition}
\begin{proof} A point in the Hilbert scheme $(\mathbb{C}^k)^{[n]}$ is defined by an ideal $I \subset \mathbb{C}[z_1,\dots,z_k]$ of length $n$, i.e. such that $\dim \mathbb{C}[z_1,\dots,z_k] / I = n$. Multiplication by $z_i$, $i=1,\dots,k$, defines an endomorphism $A_i$ of $\mathbb{C}[z_1,\dots,z_k] / I $. Clearly the $A_i$ commute and 
$\dim \cx[A_1,\dots,A_k]= \dim \mathbb{C}[z_1,\dots,z_k] / I = n$. Moreover, directly from the construction, $ \cx[A_1,\dots,A_k]$ is conjugate to its image under the regular representation. Conversely, given $k$ commuting matrices $A_1,\dots,A_k$ with $\dim \cx[A_1,\dots,A_k]=n$ we define a homomorphism $\cx[z_1,\dots,z_k]\to \cx^n\simeq \cx[A_1,\dots,A_k]$ via $p(z_1,\dots, z_k)\mapsto p(A_1,\dots,A_k)$. Clearly it is surjective and its kernel is an ideal of length $n$. Moreover this ideal does not change under simultaneous conjugation of $A_1,\dots,A_k$, and so we obtain a well-defined point of $(\mathbb{C}^k)^{[n]}$. Since we assume that $\cx[A_1,\dots,A_k]$ is conjugate to its image under the regular representation, the two maps are inverse to each other. 
\end{proof}

\begin{remark} The condition that the algebra $\sA=\cx[A_1,\dots,A_k]$ is conjugate to its image under the regular representation is equivalent to the existence of a cyclic vector for $\sA$. This is the same argument as in \cite[p.8]{Nak}. Also, the set of $k$-tuples $(A_1,\dots,A_k)$ of commuting $n\times n$ matrices with $\dim \cx[A_1,\dots,A_k]=n$ and having a cyclic vector is open in the set of all commuting $k$-tuples. Indeed, due to the Cayley-Hamilton theorem, any element of $\cx[A_1,\dots,A_k]$ belongs to the linear span of $A_1^{i_1}\dots A_k^{i_k}$ for $i_1,\dots,i_k\leq n-1$. Thus the condition of not having cyclic vector is equivalent to 
$\dim\langle A_1^{i_1}\dots A_k^{i_k}v\,;\,i_1,\dots,i_k\leq n-1\rangle\leq n-1$ for any $v$. This is a closed condition.
\end{remark}

Henni and Jardim show that the quotient of the variety of $k$ commuting matrices together with a choice of a cyclic vector  by $GL(n,\cx)$  is a geometric quotient \cite[\S 2]{HJ}, and that the resulting scheme  is isomorphic to $(\cx^k)^{[n]}$ \cite[Cor. 4.9]{HJ}. For $k=2$, these results are due to Nakajima \cite[\S 1.2]{Nak}.
\par
We shall now  modify this description of $(\cx^2)^{[n]}$ in several ways. First of all, we shall want to eliminate the dependence on a cyclic vector. We recall the  following theorem of Neubauer and Saltman \cite{NS}:
\begin{theorem}[Neubauer-Saltman] Let $A$ and $B$ be two commuting $n\times n$ matrices. Then the following conditions are equivalent:
\begin{itemize}
\item[(1)] $\dim \cx[A,B]=n$.
\item[(2)] $\dim Z(A,B)=n$, where $Z(A,B)$ is the centraliser of the pair $(A,B)$.
\item[(3)] $(A,B)$ is a nonsingular point of the variety of commuting matrices.
\end{itemize}\end{theorem}
The equivalence between (1) and (2) will be used repeatedly throughout the paper.

Let us write $M_n$ for the variety of pairs of commuting $n\times n$ matrices and define
$$ M_n^0=\bigl\{(A,B);\; A,B\in \Mat_{n\times n}(\cx),\enskip [A,B]=0,\enskip \dim Z(A,B)=n\bigr\},$$
and its open subset $M_n^{\rm reg}$ consisting of those $(A,B)\in M_n^0$ for which $\cx[A,B]$ is conjugate to its image in $\Mat_{n\times n}(\cx)$ under the regular representation.

\begin{proposition} $(\cx^2)^{[n]}$ is the geometric quotient of $M_n^{\rm reg}$ by $GL(n,\cx)$.\label{sch-iso}\end{proposition}
\begin{proof} It is the matter of checking the conditions that a geometric quotient has to satisfy \cite[Def. 0.6]{MFK}. The only one which perhaps requires an argument is that the structure sheaf of $(\cx^2)^{[n]}$ is equal to the $GL(n,\cx)$-invariant part of the structure sheaf of $M_n^{\rm reg}$. Nakajima \cite{Nak} shows that the Hilbert scheme of $n$ points in $\cx^2$ is isomorphic to the geometric quotient of the (smooth) variety
$$\tilde H_n=\{(A,B,v);\, [A,B]=0,\enskip \text{$v$ is a cyclic vector for $(A,B)$}\}$$
by the (free) action of $GL(n,\cx)$. Hence $\sO_{(\cx^2)^{[n]}}=(\sO_{\tilde H_n})^{GL(n,\cx)}$.
\par
We have the natural forgetful map $p:\tilde H_n\to M_n^{\rm reg}$, $(A,B,v)\mapsto (A,B)$, and the proof of Theorem 1.9 in \cite{Nak} shows that $p$ is a submersion. Observe now that the fibre $p^{-1}(A,B)$ is isomorphic to the stabiliser of $(A,B)$ in $GL(n,\cx)$. Indeed, the stabiliser is $n$-dimensional and it acts freely on the fibre. Since, owing to Proposition \ref{bijection}, the projection $\tilde H_n\to (\cx^2)^{[n]}$ factors (set-theoretically) through $M_n^{\rm reg}$, the action of the stabiliser on the fibre must be transitive. Therefore, for any open subset $U\subset M_n^{\rm reg}$, $\sO(U)^{GL(n,\cx)}= \sO(p^{-1}(U))^{GL(n,\cx)}$, and so $\sO_{M_n^{\rm reg}}=(\sO_{\tilde H_n})^{GL(n,\cx)}$.
\end{proof}
We shall now present another description of $(\cx^2)^{[n]}$: as a symplectic quotient of an open dense subset of pairs of matrices by a nonreductive group. The relevant subset $K_n$ consists of pairs $(A,B)$ of  $n\times n$ matrices, such that  the vector $e_1=(1,0,\dots,0)^T$ is cyclic for the pair $A,B$. The group $G$ is the subgroup of $GL(n,\cx)$ preserving the cyclic vector $e_1$, i.e. the first column of  elements of $G$ is equal to $e_1$. The symplectic form on $K_n$ is $\tr dA\wedge dB$, and the corresponding moment map $\mu:K_n\to \fG^\ast$ is 
the projection of $[A,B]$ onto the last $n-1$ rows.
\begin{theorem} The action of $G$ on $\mu^{-1}(0)\subset K_n$ is free and proper, and the symplectic quotient $\mu^{-1}(0)/G$ is biholomorphic to $(\cx^2)^{[n]}$.\label{G0}
\end{theorem}
\begin{proof} First of all, we claim that $\mu(A,B)=0$ implies $[A,B]=0$. Indeed, $\mu(A,B)=0$ means that $[A,B]$ can have nonzero entries only in the first row. Suppose that $[A,B]\neq 0$. Then 
  $\im [A,B]=\langle e_1\rangle$ is a $1$-dimensional subspace which is cyclic for $(A,B)$. On the other hand, according to \cite{Gu} (see also  \cite[Lemma 12.7]{EG} for a simple proof) $A$ and $B$ can be simultaneously conjugated into upper-triangular matrices.   But then any vector in $\im [A,B]$ has the last coordinate equal to zero and since no such vector can be cyclic for a pair of upper-triangular matrices, we obtain a contradiction. Hence $[A,B]=0$.
  \par
  The action of $G$ on $\mu^{-1}(0)$ is free, since the action of $GL(n,\cx)$ on $\tilde H_n$ is free. This implies, in particular, that $\mu^{-1}(0)$ is smooth. To show that the action of $G$ on $\mu^{-1}(0)$ is proper it is enough to show that the action of  $GL(n,\cx)$ on $\tilde H_n$ is proper, since $K_n$ is closed in $\tilde H_n$ and $G$ is closed in $GL(n,\cx)$. The properness of the  $GL(n,\cx)$ on $\tilde H_n$ is equivalent to  $\tilde H_n$ being a principal  $GL(n,\cx)$-bundle over $(\cx^2)^{[n]}$. Nakajima \cite[Theorem 3.24 \& Cor.\ 3.42]{Nak} shows that $(\cx^2)^{[n]}$ is a (real) symplectic quotient $\nu^{-1}(c)/U(n)$ of $\tilde H_n$, where $\nu$ denotes the moment map. Since $U(n)$ is compact, $\nu^{-1}(c)$ is a principal $U(n)$-bundle over $(\cx^2)^{[n]}$. This means that $\nu^{-1}(c)\to (\cx^2)^{[n]}$ admits local sections. Such a local section gives a local section of $\tilde H_n\to (\cx^2)^{[n]}$, and, consequently, $\tilde H_n$ is a principal  $GL(n,\cx)$-bundle over $(\cx^2)^{[n]}$. As explained above, this implies that the action of $G$ on $\mu^{-1}(0)$ is proper. It follows that $\mu^{-1}(0)/G$  is a complex manifold \cite[Ch.III, Prop.10]{Bour}, biholomorphic to $\tilde H_n/GL(n,\cx)$.
\end{proof}

\subsection{$(\cx^2)^{[n]}$ and torsion-free sheaves on $\oP^2$\label{notlinear}}
We shall now give a description of an open subset of $(\cx^2)^{[n]}$ consisting of $0$-dimensional subschemes not contained in any line (thus, necessarily, $n\geq 3$). 
It is closely related to  the description of the moduli space $\sM(2,n-1)$ of framed torsion-free sheaves on $\oP^2$ of rank $2$ and $c_2=n-1$ in terms of the ADHM equations \cite[Theorem 2.1]{Nak}. In order to simplify the notation, set $k=n-1$. The moduli space $\sM(2,k)$ is biholomorphic to the $GL(k,\cx)$-quotient of the set $U$ of stable solutions to the equation
\begin{equation} [X,Y]+ij=0,\quad X,Y\in \Mat_{k,k}(\cx),\enskip i\in \Mat_{k,2}(\cx),\enskip j\in \Mat_{2,k}(\cx).\label{ADHM}\end{equation}
Here a quadruple $(X,Y,i,j)$ is called  {\em stable} if there is no proper subspace $S$ of $\cx^k$ such that $\im i\subset S$, $XS\subset S$, $YS\subset S$, and the framing of a sheaf $\sF$ is a trivialisation  on the line $l_\infty\subset \oP^2$ (in particular $c_1=0$).
\begin{lemma} Let $\sF\in \sM(2,k)$. The following conditions are equivalent:
\begin{itemize}
\item[(i)] the rank of $i$ is equal to $1$;
\item[(ii)] $H^0(\oP^2,\sF)\neq 0$;
\item[(iii)]  $\sF$ is an extension of the form
$ 0\to \sO_{\oP^2}\to \sF\to \sI_Z\to 0,$ where $\sI_Z$ is the ideal sheaf of $Z\in \bigl(\oP^2\backslash l_\infty\bigr)^{[k]}$.
\end{itemize}
\end{lemma}
\begin{proof} Given $c_2(\sF)=k$, the conditions (ii) and (iii) are clearly equivalent. We prove the equivalence of (i) and (iii). If $i$ has rank $1$, then $ij=i_0(\alpha_1 j_1 +\alpha_2 j_2)$ for some $i_0\in \cx^k$, $\alpha_1,\alpha_2\in \cx$, where $j_1,j_2$ are the two rows of $j$. Equation \eqref{ADHM} implies that $[X,Y]$ has rank $1$, and therefore $\alpha_1 j_1 +\alpha_2 j_2=0$ \cite[Prop. 2.8]{Nak}. Thus $j$ has also rank $1$ and we can write $\cx^2\simeq W_1\oplus W_2$, where $W_1=\im j=\Ker i$. The sheaf $\sF$ is then isomorphic to  $\Ker b/\im a$, where $(a,b)$ is the monad given on p.\ 23 in \cite{Nak}. The embedding $W_1\otimes \sO_{\oP^2}\hookrightarrow \Ker b$ induces an injective morphism $\sO_{\oP^2}\to \sF$.
Its cokernel is the sheaf  obtained from the monad with $W$ replaced by $W_2$ and $j=0$. This is a rank one framed torsion free sheaf with $c_2=k$, i.e.\ the ideal sheaf of a $Z\in \bigl(\oP^2\backslash l_\infty\bigr)^{[k]}$. Conversely, given $\sF$ as in the statement, we obtain from $\sI_Z$ a stable solution $(X,Y,i_0,0)$ to the ADHM equation with $i_0\in \cx^k$. Let $j_0^T$ be the extension class of $\sF$ in $\Ext^1 (\sI_Z,\sO_{\oP^2})\simeq \cx^k$. Then $(X,Y,i\oplus 0,0\oplus j)$ is a solution to the ADHM equation which yields (the isomorphism class of) $\sF$.
\end{proof}

We denote by $\sM(2,k)^{\rm o}$ the complement of the (isomorphism classes of) sheaves described in this lemma. If $(X,Y,i,j)$ is the ADHM-data corresponding to a point in $\sM(2,k)^{\rm o}$, then $i$ has rank $2$ and
we can use the action of $GL(k,\cx)$ to fix $i$ to be
\begin{equation}i=\begin{pmatrix}1 & 0 & 0 &\dots & 0\\ 0 & 1 & 0 &\dots & 0\end{pmatrix}^T.\label{i}\end{equation}
The stabiliser $G_0$ of $i$ consists of matrices of the form 
$$\begin{pmatrix} 1_{2\times 2} & \ast\\ 0 &  \ast\end{pmatrix},$$
and $\sM(2,k)^{\rm o}$ is also biholomorphic to the quotient of the space $U_0$ of such $(X,Y,i,j)$ by $G_0$. Here ``quotient" means that $U_0$ is a principal $G_0$-bundle over $\sM(2,k)^{\rm o}$. Indeed, a well-known description of $\sM(2,k)$ as a hyperk\"ahler quotient means that $U$ is a principal $GL(k,\cx)$-bundle over $\sM(2,k)$ (cf. the proof of Theorem \ref{G0})
and, consequently, the action  of $GL(k,\cx)$ on $U$ is free and proper. This implies that the action of $G_0$ on $U_0$ is also free and proper, so that $U_0$ is a principal $G_0$-bundle over $\sM(2,k)^{\rm o}$. In fact, $j$ is now superfluous: it is simply given by the first two rows of $[X,Y]$. Equation \eqref{ADHM} is now equivalent to the last $k-2$ rows of $[X,Y]$ being identically zero. Since the projection onto the last $k-2$ rows is precisely the moment map $\mu$ for the action of $G_0$ on pairs of matrices $X,Y$ with respect to the symplectic form $\tr dX\wedge dY$, we conclude:
\begin{proposition} Let $V^{\rm s}$ be the set of pairs $(X,Y)\in \Mat_{k,k}(\cx)^2$ such that $(X,Y,i)$ is stable\footnote{Stability does not depend on $j$.}, where $i$ is given by \eqref{i}. The group $G_0$ acts freely and properly on $\mu^{-1}(0)\subset V^{\rm s}$  and the symplectic quotient $\mu^{-1}(0)/G_0$ is biholomorphic to $\sM(2,k)^{\rm o}$.\hfill$\Box$\label{M(2,k)}\end{proposition}
\begin{remark} Clearly, there is an analogous description of an open dense subset of $\sM(r,k)$ for sheaves of  higher rank $r$.\end{remark}
Given a quadruple $(X,Y,i,j)$ (not necessarily satisfying \eqref{ADHM} or \eqref{i}), we define a pair of $(k+1)\times (k+1)$-matrices as follows (cf. \cite{BP}):
\begin{equation}\hat X=\begin{pmatrix} 0 & -j_2\\ i_1 & X\end{pmatrix},\enskip \hat Y=\begin{pmatrix} 0 & j_1\\ i_2 & Y\end{pmatrix},\label{hatXY}\end{equation}
where $i_s$ (resp. $j_s$) denotes the $s$-th column (resp. $s$-th row) of $i$ (resp. of $j$).
\begin{lemma} $(X,Y,i,j)$ satisfies \eqref{ADHM} if and only if $[\hat X,\hat Y]$ has nonzero entries only in the first row or the first column.
Moreover $(X,Y,i,j)$ is stable if and only if $e_1$ is cyclic for $(\hat X,\hat Y)$.\label{equiv}\end{lemma}
\begin{proof} The first statement is obvious. For the second one, observe that if $S$ is a destabilising subspace for $(X,Y,i,j)$, then $S\oplus\langle e_1\rangle$ is invariant for $(\hat X,\hat Y)$. Conversely, suppose that $\hat S$ contains $e_1$ and is  invariant for $(\hat X,\hat Y)$. Let $p:\cx^{k+1}\to \cx^{k}=\cx^{k+1}/\langle e_1\rangle$ be the projection and set $S=p(\hat S)$. Then $i_1=p(\hat X e_1)$, $i_2=p(\hat Y e_1)$ belong to $S$. Moreover, if $v\in S$, then there is an $\alpha\in\cx$ such that $\alpha e_1+v\in \hat S$. Since 
$$\hat X(\alpha e_1+v)=\alpha i_1 +Xv \mod\langle e_1\rangle $$ and similarly for $Y$, it follows that $XS\subset S$ and $YS\subset S$. Therefore $S$ is a destabilising subspace for $(X,Y,i,j)$.
\end{proof}
\par
We now fix $i$ to be \eqref{i}. Let $\hat G_0$ denote the subgroup of $GL(k+1,\cx)$, consisting of matrices of the form
$$ \begin{pmatrix} 1 & u\\ 0 & h\end{pmatrix},\quad h\in G_0, \enskip u_1=u_2=0,$$
i.e. matrices, the first $3$ columns of which are $(e_1,e_2,e_3)$.
$\hat G_0$ is a semidirect product of $G_0$ and $\cx^{k-2}$. Its action by conjugation on $\hat X,\hat Y$ defined in \eqref{hatXY} is given by
$$ \hat X\mapsto \begin{pmatrix} 0 & -j_2h^{-1}+uXh^{-1}\\ i_1 & -i_1uh^{-1}+hXh^{-1}\end{pmatrix},\quad \hat Y\mapsto \begin{pmatrix} 0 & j_1h^{-1}+uYh^{-1}\\ i_2 & -i_2uh^{-1}+hYh^{-1}\end{pmatrix}.$$
This means that $\hat G_0$ acts on pairs of matrices $X,Y$ via
\begin{equation} (X,Y)\mapsto (-i_1uh^{-1}+hXh^{-1},-i_2uh^{-1}+hYh^{-1}).\label{action}\end{equation}
Owing to Lemma \ref{equiv}, $\hat G_0$ preserves $V^{\rm s}$ of Proposition \ref{M(2,k)}. 
\par
We consider the following codimension $2$ subset of $V^{\rm s}$:
$$V^{\rm s}_0=\{(X,Y)\in V^{\rm s}; X_{12}=Y_{11},\enskip X_{22}=Y_{21}\}.$$
It is $\hat G^0$-invariant and the restriction of the symplectic form $\tr dX\wedge dY$ to $V_0^{\rm s}$ is nondegenerate. 
\begin{theorem} Let $\hat\mu$ denote the moment map for the action of $\hat G^0$ on $V^{\rm s}_0$. The action of $\hat G^0$  on $\hat\mu^{-1}(0)$ is free and proper and the symplectic quotient $\hat\mu^{-1}(0)/\hat G_0$ is biholomorphic to the open subset of  $(\cx^2)^{[k+1]}$ consisting of $0$-dimensional subschemes not contained in any line.
\label{nonlin}\end{theorem} 
\begin{proof} 
The moment map for the action of $\hat G_0\simeq G_0\ltimes \cx^{k-2}$ is equal to $ \hat\mu(X,Y)=(\mu(X,Y), \alpha(X,Y))$, where
$$\alpha(X,Y)=(X_{32}-Y_{31},\dots,X_{k2}- Y_{k1})^T.$$ 
Thus, given the definition of $V^{\rm s}_0$, the condition $\alpha(X,Y)=0$ means that the first column of $Y$ is equal to the second column of $X$. This in turn implies that the first column of $[\hat X,\hat Y]$ is equal to $0$. On the other hand   $\mu(X,Y)=0$ means that the last $k-2$ rows of $[X,Y]$ are zero. If we now take $j_1$ to be the first row of $[Y,X]$ and $j_2$  the second row of $[X,Y]$, then $[\hat X,\hat Y]$ has nonzero entries only in the first row. The argument in the proof of Theorem \ref{G0} together with Lemma \ref{equiv} imply that  $[\hat X,\hat Y]=0$. Let $K_{k+1}^0$ be the set of pairs $(A,B)\in \Mat_{k+1,k+1}(\cx)^2$ such that $e_1$ is cyclic for $(A,B)$, the first column of $A$ is equal to $e_2$, and the first column of $B$ is equal to $e_3$. The above discussion shows that the map $V^{\rm s}_0\to K_{k+1}^0$,  sending $(X,Y)$ to the matrices $\hat X,\hat Y$ defined above, restricts to a $\hat G_0$-equivariant isomorphism between $ \hat\mu^{-1}(0)$ and $T=\{[A,B]\in K_{k+1}^0;\, [A,B]=0\}$. The action of $\hat G_0$ on $T$ is free and proper, owing to the same argument as in the proof of Theorem \ref{G0}. It remains to show that $T/ \hat G_0$ is biholomorphic to the open subset of  $(\cx^2)^{[k+1]}$ consisting of $0$-dimensional subschemes not contained in any line.
\par
Let $H^0\subset (\cx^2)^{[k+1]}$ be the set of $0$-dimensional subschemes  not contained in any line. If $D\in H^0$ has ideal $I$,  then $1,z_1,z_2$ are linearly independent elements of $\cx[z_1,z_2]/I$. Therefore we can choose a basis $\cx[z_1,z_2]/I$ such that $1,z_1,z_2$ are its first $3$ elements. This means that the resulting commuting matrices $A,B$ belong to $K_n^0$. The group preserving  $K_n^0\subset K_n$ is precisely $ \hat G_0$, so that $H^0\subset T/ \hat G_0$. Conversely, let $D$ be a $0$-dimensional subscheme of $\cx^2$ represented by an element of $ T/ \hat G_0$. The ideal of $D$ is given by polynomials $p(z_1,z_2)$ vanishing on $(A,B)$. Given the form of the first column of $A$ and of $B$,  no linear polynomial vanishes on $(A,B)$.
\end{proof}

\section{Curves and commuting matrices}

We now wish to describe a correspondence between finite coverings of $\oP^1$ and certain commuting matrix polynomials.

\subsection{Flat projections}
Let $C$ be a  connected curve (cf.\ Remark \ref{curve}) of arithmetic genus $g$ and $\pi:C\to \oP^1$ a flat projection of degree $d$. We consider the sheaf of algebras $\pi_\ast \sO_C$. As a sheaf of $\sO_{\oP^1}$-modules it is locally free, i.e. a vector bundle of rank $d$, which we denote by $E_\pi$. It has a trivial summand and all remaining summands have a negative degree. Moreover the degree of $E_\pi$ is equal to $-(d+g-1)$ (since $\chi(E_\pi)=\chi(\sO_C)=1-g$). Consider now the regular representation of $\pi_\ast \sO_C$ on itself. It gives a global injective morphism $\pi_\ast \sO_C\to \sEnd(E_\pi)$, i.e. a global section of $\sEnd(E_\pi)\otimes E_\pi^\ast$. Suppose that
\begin{equation} E_\pi\simeq  \bigoplus_{i=1}^{d}\sO_{\oP^1}(-k_i),\label{Epi}\end{equation}
where $0= k_1<k_2\leq \dots\leq k_{d-1}$. Then a global section of $\sEnd(E_\pi)\otimes E_\pi^\ast$ corresponds to $d$ $d\times d$ matrices $A_1,\dots, A_{d}$ of polynomials in one variable with $A_1=1$ and the $(i,j)$ entry of $A_l$ having degree $k_j-k_i+k_l$. Moreover these matrices commute, since $ \sO_C$ is commutative.  Finally, since each stalk of $\pi_\ast \sO_C$ is a $d$-dimensional algebra, we have $\dim \cx[A_2(t),\dots, A_{d}(t)]=d$ for each $t$. Clearly the freedom in choosing the matrices $A_l$ is equivalent to the choice of the isomorphism \eqref{Epi}.

Conversely, let $R$ be such a $d$-dimensional commutative subalgebra of $\Mat_{d\times d}(\cx[t])$ with identity. Then $R$ is integral over $\cx[t]$. Consider the matrices $A^\prime_l$ over $\cx[1/t]$ obtained by conjugating $A_l$ by $\diag(t^{k_1},\dots,t^{k_{d}})$ and dividing by $t^{k_l}$. We obtain a  commutative ring $R^\prime$, integral over $\cx[1/t]$. The affine curves $\Spec R$ and $\Spec R^\prime$ have an obvious gluing and we obtain a curve $C$  together with a projection $C\to \oP^1=\Spec \cx[t]\cup \Spec \cx[1/t]$. Each fibre $C_t$ is a $0$-dimensional scheme of length $d$ and so the projection $C\to \oP^1$ is flat. It is clear that the constructions are inverse to each other.

\begin{proposition} There exists a natural $1-1$ correspondence between:

\vspace{2mm}

{\rm ({\bf A})} connected  curves $C$  with a flat projection $\pi$ of degree $d$ onto $\oP^1$ and a fixed isomorphism
$$\pi_\ast\sO_C\simeq \sO_{\oP^1}\oplus \bigoplus_{i=2}^{d}\sO_{\oP^1}(-k_i),$$

\vspace{2mm}

{\rm ({\bf B})} $d-1$-tuples of commuting $d\times d$ matrix polynomials $A_2(t),\dots,A_d(t)$ such that
\begin{itemize}
\item[(i)] the first column of $A_l$ is the constant vector $e_l$;
\item[(ii)] the $(i,j)$ entry of $A_l$ has degree $k_j-k_i+k_l$;
\item[(iii)] $\dim \cx[A_2(t),\dots,A_{d}(t)]=d$ for any $t\in \oP^1$.\hfill $\Box$
\end{itemize}
\end{proposition}

\subsection{Curves in projective spaces\label{projective}} 
Suppose now that $C$ is a connected curve of degree $d$ in $\oP^r\backslash \oP^{r-2}$, not contained in any hyperplane, and let the projection $\pi$ be the restriction of the projection 

\vspace{-2.5mm}

$$[z_1,\dots,z_{r+1}]\mapsto [z_r,z_{r+1}].$$

\vspace{0.3mm}

Each $z_i$, $i=1,\dots,r-1$, defines a direct summand of $\pi_\ast\sO_C$, isomorphic to $\sO_{\oP^1}(-1)$. Moreover, since $C$ is not contained in a hyperplane, these summands are linearly independent, and, consequently, $k_2=\dots=k_r=1$ in \eqref{Epi}.
 Since $\sO_{C_t}$ is generated by $z_1,\dots,z_{r-1}$, the matrices $A_2(t),\dots,A_{r}(t)$ generate $\cx[A_2(t),\dots, A_{d}(t)]$ for each $t$.  We conclude:
\begin{proposition} There exists a natural $1-1$ correspondence between:

\vspace{2mm}

{\rm ({\bf A})} pairs $(C,\phi)$, where $C$ is a connected curve $C$ of degree $d$ in $\oP^r\backslash \oP^{r-2}$ not contained in any hyperplane and such that the projection $\pi$ onto the complementary $\oP^1$ is flat, while $\phi$ is a fixed isomorphism

\vspace{-2mm}

$$\pi_\ast\sO_C\simeq \sO\oplus \sO(-1)^{\oplus r-1}\oplus \bigoplus_{i=r+1}^{d}\sO(-k_i),$$

\vspace{-4mm}

and

\vspace{2mm}

{\rm ({\bf B})} $(r-1)$-tuples of commuting $d\times d$ matrix polynomials $A_2(t),\dots, A_{r}(t)$ such that
\begin{itemize}
\item[(i)] the first column of $A_l$ is the constant vector $e_l$, $l=2,\dots,r$;
\item[(ii)] the $(ij)$-entry of $A_l$ has degree $k_j-k_i+k_l$ (here $k_1=0$ and $k_2=\dots=k_r=1$); 
\item[(iii)] for any $t\in \oP^1$, $\cx[A_2(t),\dots,A_{r}(t)]$ has dimension $d$ and is conjugate to its image in $\End(\cx^d)$ under the regular representation.\hfill $\Box$
\end{itemize}\label{correspondence}
\end{proposition}

\begin{example} Let $C$ be a smooth curve of degree $r$ in $\oP^r\backslash \oP^{r-2}$. Such a curve is cut out by the $2\times 2$ minors of a $2\times r$ matrix of linear forms in homogeneous coordinates $z_1,\dots,z_{r+1}$. Thus we obtain $\begin{pmatrix} r\\ 2\end{pmatrix}$ quadratic equations. Since we assume that the projection onto $[z_r,z_{r+1}]$ is flat, the matrix of coefficients of $z_iz_j$, $i,j\leq r-1$, is invertible and we can write the equations in affine coordinates $t=z_r/z_{r+1}$, $x_i=z_i/z_{r+1}$, $i\leq r-1$,
as
$$ x_ix_j=\sum_{k=1}^{r-1}a_{ij}^k(t)x_k+b_{ij}(t),\quad i,j=1,\dots r-1,$$
where $a_{ij}$ are linear in $t$ and $b_{ij}$ is quadratic in $t$. It follows that the commuting matrices $A_2(t),\dots, A_{r}(t)$ are given by
$$A_l(t)=\begin{pmatrix} 0 & b_{l1}(t) & \dots & b_{l,r-1}(t)\\
                           0 & a_{l1}^1(t) &\dots & a_{l,r-1}^1(t)\\
                           \vdots & \vdots &&\vdots \\
                           1 & a_{l1}^l(t) &\dots & a_{l,r-1}^l(t)\\
                           \vdots & \vdots &&\vdots \\
                           0 & a_{l1}^{r-1}(t) &\dots & a_{l,r-1}^{r-1}(t)\\
\end{pmatrix}.
$$
\label{canonical}\end{example}

\vspace{2mm} 

\begin{example} The correspondence in Proposition \ref{correspondence} commutes with projections. Thus, if $C$ is a curve of degree $d$ in $\oP^r\backslash \oP^{r-2}$ obtained by projection onto the first $r$ coordinates from a curve $\hat C$ in $\oP^n\backslash \oP^{n-2}$, then the matrix polynomials $A_2(t),\dots,A_{r}(t)$ corresponding to the curve $C$ are simply the first $r-1$ matrix polynomials corresponding to the curve $\hat C$.\label{proj}
\end{example}

\begin{example} Let $C$ be a rational curve of degree $4$ in $\oP^3$ parametrised by $[u^3v,v^3u,u^4,v^4]$. We can either use the previous example or proceeed directly as follows: if we set $x=z_1/z_4,y=z_2/z_4,t=z_3/z_4$, then the ideal of $C$ is generated by 
$$ xy-t,\enskip t^2y-x^3,\enskip ty^2-x^2, \enskip y^3-x.$$
It follows that $\sO(C_t)$ is spanned by $1,x,y,y^2$ for $t\neq \infty$ and by $1,x,y,x^2$ for $t\neq 0$.
Computing the endomorphisms $x\cdot(\;)$ and $y\cdot(\;)$ gives:
$$ A=\begin{pmatrix}  0 & 0 & t & 0\\
1 & 0 &0&0\\0&0&0&t\\0&t&0&0\end{pmatrix},\quad B=\begin{pmatrix} 0&t&0&0\\0&0&0&1\\1&0&0&0\\0&0&1&0\end{pmatrix}.$$
\end{example}

\begin{example} Let $C$  a complete intersection of two quadrics in $\oP^3$, given by equations $Q_1(x,y,t)=0$, $Q_2(x,y,t)=0$. Let $\tilde Q_1(x,y), \tilde Q_2(x,y)$ be the parts involving only $x,y$ and let $\tilde Q_3(x,y)$ be the quadratic polynomial in $x,y$ independent of $\tilde Q_1, \tilde Q_2$. Then $\sO(C_t)$ is spanned by $1,x,y,\tilde Q_3$, and hence $E_\pi\simeq \sO_{\oP^1}\oplus\sO_{\oP^1}(-1)\oplus\sO_{\oP^1}(-1)\oplus\sO_{\oP^1}(-2)$. In particular, the arithmetic genus of $C$ is $1$. Generically, such a $C$ is a smooth elliptic curve, but if $Q_1$ and $Q_2$ are two of the three quadrics defining a twisted cubic, then $C$ is the union of this cubic and a line, intersecting in two points. \label{112}
\end{example}

\section{Space curves\label{space}}


We consider the fibrewise Hilbert scheme $Z_d$ of $d$ points for the projection $\pi:\oP^3\backslash \oP^1\to \oP^1$ (see \cite[Ch. 1, \S7]{ACG} for a definition and properties of relative Hilbert schemes). Locally, over an open subset $U$ of $\oP^1$, it is just $U\times (\cx^2)^{[d]}$. It is therefore  a $(2d+1)$-dimensional complex manifold with a holomorphic projection $p:Z_d\to \oP^1$. 
\par
$Z_d$ satisfies the necessary conditions to be the twistor space of a pseudo-hyperk\"ahler manifold (i.e.\ a  pseudo-Riemannian manifold $M$ with a fibrewise action of quaternions on $TM$, parallel with respect to the Levi-Civita connection):
\begin{itemize}
\item[(i)] it has an antiholomorphic involution $\sigma$ covering the antipodal map on $\oP^1$, induced from the standard antiholomorphic involution on the total space of $\sO_{\oP^1}(1)\oplus \sO_{\oP^1}(1)\simeq \oP^3\backslash \oP^1$;
\item[(ii)]it has an $\sO_{\oP^1}(2)$-valued symplectic form $\omega$ along the fibres of $\pi$, again induced\footnote{Recall \cite{Beau} that the Hilbert scheme of $d$ points on a symplectic surface has a canonical symplectic structure.} from the $\sO_{\oP^1}(2)$-valued fibrewise symplectic form on  the total space of $\sO_{\oP^1}(1)\oplus \sO_{\oP^1}(1)$.
\end{itemize}
The pseudo-hyperk\"ahler manifold is then the Kodaira moduli space of $\sigma$-invariant sections of $p$, the normal bundle 
 of which splits as $\sO_{\oP^1}(1)^{\oplus 2d}$ \cite[\S3(F)]{HKLR}. 

 \subsection{Normal sheaf of space curves}

 \begin{proposition} There exists a natural isomorphism between an open subset of the Hilbert scheme of $\oP^3\backslash \oP^1$ consisting of degree $d$  curves which are flat over $\oP^1$, and the Hilbert scheme of sections of  $p:Z_d\to \oP^1$.\label{bihomo}
 \end{proposition}
 \begin{remark} The Hilbert scheme of curves of degree $d$ is the union of all components of $\Hilb(\oP^3\backslash \oP^1)$ which have the Hilbert polynomial of the form $h(n)=dn+c$, $c\in \oZ$. The Hilbert scheme of sections of $p:Z_d\to \oP^1$ is an open subscheme of $\Hilb(Z_d)$.\end{remark}
 \begin{proof} If $C$ is a degree $d$  curve in $\oP^3 \backslash \oP^1$, flat over $\oP^1$, then its scheme-theoretic intersection with each fibre of $\oP^3\backslash \oP^1\to \oP^1$ yields a section of $Z_d$. The inverse map is defined as follows: given a section of $Z_d$,  pullback the universal family over the relative Hilbert scheme $Z_d$ to $\oP^1$. A flat family of curves $W$ in $\oP^3 \backslash \oP^1$ parameterised by a scheme $T$ is sent to a flat family of sections of $p$ (viewed as a flat family of $\oP^1$'s in $Z_d$).  Similarly, the inverse map sends a flat family of sections to a flat family of degree $d$ curves. The functorial interpretation of Hilbert schemes implies that both maps are  morphisms of schemes.
 \end{proof}
We can relate the normal bundle $N_{s/Z_d}$ of a section to the normal sheaf $\sN_{C/\oP^3}$ of the corresponding curve as follows:
\begin{lemma} $N_{s/Z_d}\simeq \pi_\ast \sN_{C/\oP^3}$.\label{pi*}\end{lemma}
\begin{proof} The normal bundle of a section is isomorphic to the restriction of the vertical tangent bundle $\Ker dp$ of $Z_d$ to the section. Since $Z_d$ is the fibrewise Hilbert scheme of points, the fibre of the vertical tangent bundle at $D\in Z_d$ is $H^0(D,\sN_{D/F})$, where $F\simeq \oC^2$ is the fibre containing $D$. On the other hand the ideal of $D$ in $F$ is just $J_C\otimes \sO_F$, where $J_C$ is the ideal of $C$ in $\oP^3$. Thus $\sN_{D/F}=\sN_{C/\oP^3}\otimes \sO_F$, and hence the function $\oP^1\ni z\to h^0(\pi^{-1}(z),\sN_{C/\oP^3}|_{\pi^{-1}(z)})$ is constant. It follows from a result of Grauert \cite[Cor.\ III.12.9]{Hart0} that $ \pi_\ast \sN_{C/\oP^3}$ is locally free. Therefore both $ \pi_\ast \sN_{C/\oP^3}$ and $N_{s/Z_d}$  are vector bundles and the natural map $ \pi_\ast \sN_{C/\oP^3}\to N_{s/Z_d}$ is an isomorphism on fibres.
\end{proof}
\begin{corollary} The normal sheaf of a Cohen-Macaulay curve in $\oP^3$ is torsion-free. \label{CMTF}
\end{corollary} 

\vspace{-6mm}

\begin{proof} We can find a projective line $l\subset \oP^3$ such that the projection $\pi:C\to l$ is flat. Were $\sN_{C/\oP^3}$ not torsion free, neither would be  $ \pi_\ast \sN_{C/\oP^3}$, contradicting the above lemma.
\end{proof}

Lemma \ref{pi*} implies that the normal bundle of a section splits as $\sO_{\oP^1}(1)^{\oplus 2d}$ if and only if $H^\ast(C, \sN_{C/\oP^3}(-2))=0$. As mentioned at the beginning of the section, the parameter space of curves satisfying this condition has a natural complexified pseudo-hyperk\"ahler structure and on its $\sigma$-invariant part a genuine pseudo-hyperk\"ahler structure \cite{sigma}. We now want to investigate this parameter space via commuting matrix polynomials.
\par
If $E$ is a rank $d$ vector bundle on $\oP^1$, we denote by $\sK(E)$ the subsheaf of  $\bigl(\sEnd(E)\otimes \sO_{\oP^1}(1)\bigr)^{\oplus 2}$ defined by
$$ \sK(E)(U)=\bigl\{(A(t),B(t))\in \bigl(\sEnd(E)\otimes \sO_{\oP^1}(1)\bigr)^{\oplus 2}(U);\, \forall_{t\in U}\enskip (A(t),B(t))\in M_d^{\rm reg} \bigr\}.$$
We then denote by $K(E)$ the total space of $\sK(E)$, i.e.\ the subset of the total space of $\bigl(\sEnd(E)\otimes \sO_{\oP^1}(1)\bigr)^{\oplus 2}$ consisting of points through which passes a local section of $\sK(E)$.
Since $Z_d$ is the relative Hilbert scheme for the projection $\sO_{\oP^1}(1)\oplus \sO_{\oP^1}(1)\to \oP^1$,
 Proposition \ref{sch-iso} implies that $Z_d$ is biholomorphic to the quotient of $K\bigl(\sO_{\oP^1}^{\oplus d}\bigr)$
by the fibrewise action of $GL(d,\cx)$. In fact we have:
\begin{lemma} Let $E$ be a vector bundle of rank $d$ on $\oP^1$. Then the quotient of $K(E)$
by the fibrewise action of $GL(d,\cx)$ is biholomorphic to $Z_d$. \label{qKE}
\end{lemma}
\begin{proof} Let $E\simeq \bigoplus\sO_{\oP^1}(\lambda_i)$. Then $\sEnd(E)\otimes \sO_{\oP^1}(1)$ is isomorphic to two copies $\cx\times \Mat_{d\times d}(\cx)$ glued over $t\neq0,\infty$ by $(\tilde t,\tilde A)=(1/t, t^{\lambda}At^{-\lambda})$, where $\lambda=\diag(\lambda_1,\dots,\lambda_d)$. Thus, after  taking the  fibrewise quotient by $GL(d,\cx)$, we obtain the same complex manifold, independently of the choice of $E$.
\end{proof}

From now on $d\geq 3$.
Let $C$ be a connected nonplanar curve of degree $d$ in $\oP^3\backslash \oP^1$,  flat over $\oP^1$. If $C$ satisfies $\pi_\ast \sO_C\simeq E$, then the corresponding section of $p:Z_d\to \oP^1$ arises as the projection of a global section of $\sK(E)\to\oP^1$, i.e. it can be represented by a pair of commuting matrix polynomials, the degrees of which satisfy the constraints of part B in Proposition \ref{correspondence}.  We should like to remark that this gives constraints for possible degrees of commuting pairs of polynomials $A(t),B(t)$, such that $\dim Z(A(t),B(t))=d$ for each $t$.

Let $\sZ$ denote the bundle of centralisers of $(A,B)$, i.e. a subbundle of $\sEnd(E)$, the fibre of which over each $t$ is spanned by $Z(A(t),B(t))$. Let $\sT$ denote the kernel of the homomorphism 
\begin{equation} D:\bigl(\sEnd(E)\otimes \sO_{\oP^1}(1)\bigr)^{\oplus 2}\to \sEnd(E)\otimes \sO_{\oP^1}(2), \quad (a,b)\to  [A,b ]+[a,B].\label{sT}\end{equation}
In other words $\sT$ is the fibrewise tangent bundle to $K(E)$. We obtain a short exact sequence of locally free sheaves on $\oP^1$:
\begin{equation} 0\to\sEnd(E)/\sZ\to \sT\to N\to 0,\label{TtoN}
\end{equation}
where $N$ is the normal bundle of the section in $Z_d$ and the first map is given by $\rho\mapsto \bigl([\rho,A],[\rho,B]\bigr)$, i.e.\ its image consists of fundamental vector fields for the action of $GL(d,\cx)$.
\begin{lemma} $\sZ\simeq E$ and if we write $E\simeq \sO_{\oP^1}\oplus E^\prime$, then the embedding $\sZ\hookrightarrow \sEnd(E)\simeq \sO_{\oP^1}\oplus E^\prime\oplus(E^\prime)^\ast\oplus \sEnd(E^\prime)$ is the isomorphism with the direct sum of the first two summands.
\end{lemma}
\begin{proof} Let $C$ be the curve in $\oP^3\backslash \oP^1$ determined by $A(t),B(t)$. We know that $\pi_\ast\sO_C\simeq E$. Let $E^\prime \simeq \bigoplus_{i=2}^d \sO_{\oP^1}(\lambda_i)$. Each summand defines a section $A_i(t)$ of $\sEnd(E)\otimes\sO_{\oP^1}(-\lambda_i)$ and these matrix polynomials commute. Moreover the first column of each $A_i(t)$ is just $e_i$. The bundle $\sZ$ is spanned over each $t$ by $1$ and the $A_i(t)$, i.e. it is isomorphic to the first column of $\sEnd(E)$. 
The claim follows.
\end{proof}

The image sheaf of $D$ is also locally free and can be identified with the annihilator of $\sZ$ in $\sEnd(E)\otimes \sO_{\oP^1}(2)$ (as vector bundles):
$$ \rho \in\im D \iff \tr\rho z=0\enskip \forall z\in \sZ.$$ 
Let us write $\sL=(E^\prime)^\ast\oplus \sEnd(E^\prime)\simeq \sEnd(E)/\sZ$. As a matrix of endomorphisms of $E$, it has the first column equal to $0$. From the above characterisation of $\im D$ it follows that $\im D\simeq \sL^\ast(2)$, and, consequently, we can write the two short exact sequences (i.e. \eqref{TtoN} and the sequence induced by \eqref{sT}) as:
\begin{equation} 0\to \sL\to \sT\to N\to 0,\label{KTN}\end{equation}
\begin{equation} 0\to \sT\to \bigl(\sEnd(E)\otimes \sO_{\oP^1}(1)\bigr)^{\oplus 2}\to \sL^\ast(2)\to 0.\label{K(2)}\end{equation}
 These sequence describe $Z_d$, at least infinitesimally, as a fibrewise symplectic quotient of the total space of $\bigl(\sEnd(E)\otimes \sO_{\oP^1}(1)\bigr)^{\oplus 2}$ by a Lie group, the Lie algebra of which is $\sL$. Our idea, on how to study the hyperk\"ahler structure of the Hilbert scheme of curves in $\oP^3\backslash\oP^1$, is essentially to describe the hyperk\"ahler manifold of sections of $Z_d$ as a ``hyperk\"ahler quotient" of the (open dense subset of) vector space $H^0\bigl(\oP^1,\bigl(\sEnd(E)\otimes \sO_{\oP^1}(1)\bigr)^{\oplus 2}\bigr)$ by a Lie group, the Lie algebra of which is $H^0(\oP^1,\sL)$. This cannot quite work in this form for various reasons, the most obvious of which is that the sections of $\bigl(\sEnd(E)\otimes \sO_{\oP^1}(1)\bigr)^{\oplus 2}$ have the wrong normal bundle in order to form a hyperk\"ahler manifold. In \S\ref{rational} we shall show how to modify and  implement this idea for  curves of genus $0$. For higher genera, one probably has to work with the $\sP$-structures of Gindikin \cite{Gi}, and with their quotients.

\medskip

\begin{remark} The above sequences provide numerical restrictions on curves which are cohomologically stable. Indeed, 
tensor \eqref{KTN} and \eqref{K(2)} with $\sO(-2)$ and put:
$$ a=h^0((\sL(-2))=h^1(\sL^\ast),\enskip b=h^1((\sL(-2))=h^0(\sL^\ast),$$
$$c=h^0(\sEnd(E)(-1))=h^1(\sEnd(E)(-1)).$$
Then $h^0(\sT(-2))\geq 2c-b$ and, consequently, if  $C$ is cohomologically stable, i.e. $h^0(N(-2))=h^1(N(-2))=0$, then $a\geq 2c-b$.
Let
$$\pi_\ast\sO_C\simeq \sO_{\oP^1}\oplus\bigoplus_{i=1}^s\sO_{\oP^1}(-i)^{m_i}.$$
Then $a=\sum_{j\geq i+2} (j-i-1)m_im_j+\sum_{i\geq 2}(i-1)m_i$, $b=\sum_{j\geq i} (j-i+1)m_im_j$, $c=\sum_{j\geq i+1} (j-i)m_im_j+\sum_{i} im_i$. Thus
$$ 0\geq 2c-a-b=\sum_i(i+1)m_i-\sum_{i}m_i^2.$$
For example a curve such that $m_i\leq i+1$ for all $i$ and $m_i\leq i$ for at least one $i$ cannot be cohomologically stable (cf. Example \ref{112}). 
\end{remark} 


\subsection{Twistor space revisited\label{revisit}}
Given the constraints of Proposition \ref{correspondence} on the matrices $A,B$ representing a space curve, we can  replace $Z_d$ for $d\geq 3$ by its open dense subset $Z^\prime_d$, the fibres of which consist of $0$-dimensional subschemes of $\cx^2$ not contained in any line. We have an analogue of Proposition \ref{bihomo}, the proof of which is the same:
\begin{proposition} Let $d\geq 3$. There exists a natural isomorphism between an open subset of the Hilbert scheme of $\oP^3\backslash \oP^1$ consisting of degree $d$   curves which are flat over $\oP^1$, and the Hilbert scheme of sections of  $p:Z_d^\prime\to \oP^1$.\hfill $\Box$\label{bihom}
 \end{proposition}
Theorem \ref{nonlin} tells us how to construct $Z^\prime_d$ from a vector bundle on $\oP^1$.
Let $E^\prime$ be a vector bundle of rank $d-1$ over $\oP^1$ of the form $E^\prime\simeq \bigoplus_{i=1}^{d-1}\sO_{\oP^1}(-k_i)$ with $k_1=k_2=1$ and all $k_i$ positive.
Define $\sK^0(E^\prime)$ to be the subsheaf of $\bigl(\sEnd(E^\prime)\otimes \sO_{\oP^1}(1)\bigr)^{\oplus 2}$, the local section of which are pairs $(X(t),Y(t))$ such that for all $t$:
\begin{itemize}
\item[(i)] the first column of $Y(t)$ is equal to the second column of $X(t)$;
\item[(ii)] there is no subspace $S$ of $\cx^{d-1}$ such that $e_1,e_2\in S$, $X(t)S\subset S$, $Y(t)S\subset S$;
\item[(iii)]  $[X(t),Y(t)]$ has nonzero entries only in the first two rows.
\end{itemize}
We denote by $K^0(E)$ the total space of $\sK^0(E)$ (defined as for $K(E)$).
We define $Z(E^\prime)$ to be the fibrewise quotient of $K^0(E^\prime)$ by the action \eqref{action} of the group $\hat G^0$ defined in \S\ref{notlinear}. Theorem \ref{nonlin} implies that $Z(E^\prime)\simeq Z^\prime_d$ and that is a smooth manifold fibreing over $\oP^1$  and we can conclude:
\begin{proposition} $Z(E^\prime)\simeq Z_d^\prime$. \label{Z(E)}
\end{proposition}
\begin{proof} This is the same argument as in the proof of Lemma \ref{qKE}.
\end{proof}
Consequently, sections of $Z^\prime_d$ corresponding to curves with $\pi_\ast \sO_C\simeq \sO_{\oP^1}\oplus E^\prime$ arise as projections of global sections of $K(E^\prime)$, and we can represent them by pairs $(X(t),Y(t))$ of matrix polynomials satisfying the above constraints.

\section{Interlude: complexified hyperk\"ahler structures\label{interlude}}

The complexified algebra of quaternions is isomorphic to $\Mat_{2,2}(\cx)$. Consequently, a complexification of a real manifold with a geometry based on quaternions (such as hyperk\"ahler, hypercomplex, quaternionic) or on split quaternions (hypersymplectic geometry) will posses a geometry based on this algebra (called {\em algebra of biquaternions} by Hamilton). Such geometries have been considered in the past, in particular by Jardim and Verbitsky in \cite{JV}. They call a complexified hyperk\"ahler structure a {\em trisymplectic structure generating an $SL(2,\cx)$-web}. Although there are good reasons for this terminology, we prefer the shorter name {\em $\cx$-hyperk\"ahler}. 
\begin{definition} A complex manifold $M$ is called {\em almost $\cx$-hypercomplex} if its holomorphic tangent bundle $TM$ decomposes as $TM\simeq E\otimes \cx^2$, where $E$ is a holomorphic vector bundle. It is $\cx$-hypercomplex if, in addition, for any $v\in \cx^2$ the subbundle $E\otimes v$ defines an integrable distribution on $M$.
\end{definition}
\begin{definition} A $\cx$-hypercomplex manifold is called {\em $\cx$-hyperk\"ahler} if it is equipped with a holomorphic section $g$ of $S^2T^\ast M$, such that:
\begin{itemize}
\item[(i)] at any $m\in M$, $g_m$ is nondegenerate,
\item[(ii)]for any $A\in \Mat_{2,2}(\cx)$, $g(A\cdot,\cdot)=g(\cdot,A_{\rm adj}\cdot)$, and
\item[(iii)] for any $A\in \ssl(2,\cx)$, the holomorphic $2$-form $\omega_A=g(A\cdot,\cdot)$ is closed.
\end{itemize}
\end{definition}
Given a $\cx$-hypercomplex manifold we can define an integrable distribution $\sD$ on $M\times \oP^1$ by $\sD_{|M\times\{z\}}=E\otimes h$, where $h$ is the highest weight vector for the maximal torus in $SL(2,\cx)$ corresponding to $z\in \oP^1$. If $\sD$ is simple\footnote{We call a distribution $\sD$ on a manifold $Y$ {\em simple}, if the leaf space $X$ of the corresponding foliation is a manifold and the map $Y\to X$ is a submersion.}, then the leaf space $Z$ is the twistor space of $M$, and points of $M$ correspond to sections of $Z\to \oP^1$ with normal bundles spliting as $\bigoplus \sO_{\oP^1}(1)$. If the distribution $\sD$ is not simple, we need to view $Z$ in terms of foliated geometry.
\par
If $M$ is $\cx$-hyperk\"ahler, then its twistor space is equipped with a fibrewise $\sO_{\oP^1}(2)$-valued complex symplectic form.
\par
We shall now discuss $\cx$-hypercomplex and $\cx$-hyperk\"ahler quotients. In the case of a reductive Lie group, $\cx$-hyperk\"ahler quotients have been introduced and studied by Jardim and Verbitsky \cite{JV} under the name ``trisymplectic reduction". For us the case of nonreductive groups will be of paramount importance.
\begin{definition} Let $M$ be a $\cx$-hypercomplex manifold equipped with a holomorphic action of a complex Lie group $G$ preserving the $\cx$-hypercomplex structure. A {\em $\cx$-hypercomplex moment map} is a $G$-equivariant holomorphic map $\mu:M\to \g^\ast\otimes \ssl(2,\cx)^\ast$, such that there exists a holomorphic $\g^\ast$-valued  $1$-form $\phi$ with the property that $\Phi=\phi \cdot 1_{2\times2}+d\mu:TM\to \g^\ast\otimes \Mat_{2,2}(\cx)^\ast$ satisfies $\langle \Phi(Av),B\rangle=\langle \Phi(v),A_{\rm adj}B\rangle$ for any $A\in \Mat_{2,2}(\cx)$.
\end{definition}
\begin{definition} Let $M$ be a $\cx$-hyperk\"ahler manifold equipped with a holomorphic action of a complex Lie group $G$ preserving the $\cx$-hyperk\"ahler structure. A {\em $\cx$-hyperk\"ahler moment map} is a $G$-equivariant holomorphic map $\mu:M\to \g^\ast\otimes \ssl(2,\cx)^\ast$, such that, for any $A\in \ssl(2,\cx)$ and any fundamental vector field $X_\rho$, $\rho \in \g$, $\langle d\mu(\cdot),\rho\otimes A\rangle=\omega_A(X_\rho, \cdot)$.
\end{definition}
\begin{remark} A $\cx$-hyperk\"ahler moment map is also a $\cx$-hypercomplex moment map, with $\langle \phi(v),\rho\rangle=g(X_\rho,v)$, $\rho \in \g$, $v\in TM$.
\end{remark}
The $\cx$-hypercomplex or $\cx$-hyperk\"ahler reduction proceeds now along the usual lines: given $G$ and a moment map $\mu:M\to \g^\ast\otimes \ssl(2,\cx)^\ast$, choose a $G$-invariant element $c\in \g^\ast\otimes \ssl(2,\cx)^\ast$. Unlike in the hyperk\"ahler case, the freeness of the action of $G$ on $\mu^{-1}(c)$ does not imply that $\mu^{-1}(c)$ is smooth. Moreover, even if $\mu^{-1}(c)$ is smooth and the action of $G$ on $\mu^{-1}(c)$ is free and proper, then although $ \mu^{-1}(c)/G$ is a complex manifold \cite[Ch.III, Prop.10]{Bour}, it is not necessarily a $\cx$-hypercomplex or $\cx$-hyperk\"ahler manifold. As observed by several authors in related settings  (especially \cite[\S 3]{Joyce}, \cite{H}, \cite[\S 4]{DS}), both smoothness of  $\mu^{-1}(c)$ and the existence of induced geometry on $ \mu^{-1}(c)/G$ are guaranteed by a single nondegeneracy condition. Namely, we have:
\begin{theorem} Let $\mu:M\to \g^\ast\otimes \ssl(2,\cx)^\ast$ be a $\cx$-hypercomplex (resp. $\cx$-hyperk\"ahler) moment map and let $c\in \g^\ast\otimes \ssl(2,\cx)^\ast$ be $G$-invariant. Suppose that the action of $G$ on $\mu^{-1}(c)$ is free and proper and that, at any $m\in \mu^{-1}(c)$, $\phi_m(X_\rho)\neq 0$ for any $\rho\in \g$ (resp. the restriction of $g$ to the subspace of $T_mM$ generated by fundamental vector fields is nondegenerate). Then $ \mu^{-1}(c)/G$ is a $\cx$-hypercomplex (resp. $\cx$-hyperk\"ahler) manifold.\hfill $\Box$\label{nondegen}
\end{theorem}
However, the assumptions in Theorem \ref{nondegen} are not the most general ones, under which one obtains a nondegenerate $\cx$-hyperk\"ahler quotient:
\begin{proposition} Let $\mu:M\to \g^\ast\otimes \ssl(2,\cx)^\ast$ be a  $\cx$-hyperk\"ahler moment map and let $c\in \g^\ast\otimes \ssl(2,\cx)^\ast$ be $G$-invariant. Suppose that $ \mu^{-1}(c)$ is smooth and the action of $G$ on $\mu^{-1}(c)$ is free and proper. Then  $ \mu^{-1}(c)/G$ is a $\cx$-hyperk\"ahler manifold if and only if $\check\g\cap \check\g^\perp$ is $\Mat_{2,2}(\cx)$-invariant along $\mu^{-1}(c)$, where $\check\g$ denotes  the subspace of $T_mM$ generated by fundamental vector fields, and $ \check\g^\perp=\{v;g(X_\rho,v)=0\enskip \forall \rho\in \g\}$.\label{moregeneral}
\end{proposition}
\begin{proof} Let $I,J,K\in \ssl(2,\cx)$ denote the standard basis of quaternions. Hitchin \cite{H} (see also \cite[p.102]{DS}) shows that along $\mu^{-1}(c)$,
$$ \Ker \omega_I=\check\g+J\check\g\cap \check\g^\perp+K\check\g\cap \check\g^\perp,
$$ and cyclically in $I,J,K$. Thus $\omega_I$ descends to a nondegenerate symplectic form on $ \mu^{-1}(c)/G$ if and only if $J\check\g\cap \check\g^\perp+K\check\g\cap \check\g^\perp\subset \check\g$. Therefore  $ \mu^{-1}(c)/G$ is a $\cx$-hyperk\"ahler manifold if and only if $\check\g\cap \check\g^\perp$ is $\Mat_{2,2}(\cx)$-invariant.\end{proof}
\begin{proposition} Suppose that the conditions of the above proposition are satisfied. Then $\dim\check\g\cap \check\g^\perp=k$ is constant along $\mu^{-1}(c)$ and $$\dim  \mu^{-1}(c)/G=\dim M-4\dim G+2k.$$\end{proposition}
\begin{proof}
The smoothness of $\mu^{-1}(c)$ implies that  $(\Ker d\mu)^\perp=I\check\g+ J\check\g+K\check\g$ has constant dimension along  $ \mu^{-1}(c)$, where $\perp$ denotes the orthogonal ``complement" with respect to the form $g$. Suppose that $IX_{\rho_1}+JX_{\rho_2}+KX_{\rho_3}=0$. Taking the scalar products (for the form $g$) with $IX_\rho,JX_\rho,KX_\rho$ shows that  $X_{\rho_1},X_{\rho_2},X_{\rho_3}\in \check\g\cap\check\g^\perp$. Therefore $\dim (\Ker d\mu)^\perp=3\dim G-2\dim\check\g\cap \check\g^\perp$ and it follows that $\dim\check\g\cap \check\g^\perp$ must be constant along $ \mu^{-1}(c)$.\end{proof}
We shall use the above results in the following situation.
\begin{corollary} Let $\mu=\mu_H\oplus\mu_L:M\to \g^\ast\otimes \ssl(2,\cx)^\ast$ be a  $\cx$-hyperk\"ahler moment map for a semidirect product $G\simeq H\ltimes L$ and let $c=(c_H,c_L)\in \g^\ast\otimes \ssl(2,\cx)^\ast$ be $G$-invariant. Suppose that the following conditions are satisfied:
\begin{enumerate}
\item[(i)] the action of $G$ on $\mu^{-1}(c)$ is free and proper;
\item[(ii)] $\mu_L^{-1}(c_L)$  is smooth and the action of $L$ on $\mu_L^{-1}(c_L)$ is free and proper;
\item[(iii)] $\check\fL$ is $\Mat_{2,2}(\cx)$-invariant along $\mu^{-1}(c_L)$;
\item[(iv)] $\check\fL\subset \check\g\cap \check\g^\perp $ along $\mu^{-1}(c)$, with equality holding generically.
\end{enumerate}
Then:
\begin{itemize}
\item[a)] $ \mu^{-1}(c)/G$ is smooth if and only if $\dim I\check\fH+J\check\fH+K\check\fH=3\dim H$ 
along $\mu^{-1}(c)$.
\item[b)]
 $ \mu^{-1}(c)/G$ is a $\cx$-hyperk\"ahler manifold  if and only if 
$\check\g\cap \check\g^\perp=\check\fL$ along $\mu^{-1}(c)$.
\end{itemize}
In both a) and b), the dimension of  $ \mu^{-1}(c)/G$ is $\dim M-4\dim H-2\dim L$.
\label{smoothQ}\end{corollary}
\begin{proof} Conditions (ii) and (iii) together with Proposition \ref{moregeneral} imply that $M_L=\mu_L^{-1}(c_L)/L$ is a $\cx$-hyperk\"ahler manifold (of dimension $\dim M-2\dim L$). Since $L$ is normal in $G$, we obtain a $\cx$-hyperk\"ahler action of $H$ on $M_L$ and a $\cx$-hyperk\"ahler moment map $\bar\mu_H:M_L\to \fH^\ast\otimes \ssl(2,\cx)^\ast$. The $\cx$-hyperk\"ahler quotient of $M_L$ by $H$ is isomorphic to the $\cx$-hyperk\"ahler quotient of $M$ by $G$. Statement a)  follows, since $\dim I\check\fH+J\check\fH+K\check\fH=3\dim H$ is equivalent to $\bar\mu_H$ being a submersion at points of $\bar\mu_H^{-1}(c_H)$. On the other hand, the condition in b) means that
the triple $(M_L,H,\bar\mu_H)$ satisfies conditions of Theorem \ref{nondegen}, and is therefore sufficient. Conversely, if $ \mu^{-1}(c)/G$ is a $\cx$-hyperk\"ahler manifold, then, owing to a), $\dim I\check\fH+J\check\fH+K\check\fH=3\dim H$, which means that $\bar\mu_H^{-1}(c_H)$ is smooth.  
The neccessity of $\check\g\cap \check\g^\perp=\check\fL$ follows now from Proposition \ref{moregeneral}.
\end{proof}

\section{Genus zero space curves\label{rational}}

We return to the situation discussed in \S\ref{space} and consider in detail the case of genus $0$ space curves. We denote by $R_d$ an open subset of the Hilbert scheme 
$\Hilb_{d,0}$ of subschemes in $\oP^3$ with Hilbert polynomial $h(n)=dn+1$, consisting of those $C\in \Hilb_{d,0}$ which satisfy
\begin{itemize}
\item[(i)] $C$ is contained in $\oP^3\backslash\{[z_0,z_1,0,0]\}$;
\item[(ii)] the projection $\pi$ of $C$ onto $\{[0,0,z_2,z_3]\}$ is flat;
\item[(iii)] $C$  is nonplanar.
\end{itemize}
Such a $C$ is automatically Cohen-Macaulay and connected.
We denote by $R_d^{i}$, $i=0,1,2$, an open subset of  $R_d$ consisting of those $C\in \Hilb_{d,0}$ which satisfy, in addition, $h^1(\sN_{C/\oP^3}(-i))=0$. We have $R_d^2\subset R_d^1\subset R_d^0$ and $R_d^0$ is precisely the smooth locus of $R_d$. Moreover, the results of Ghione and Sacchiero \cite[Cor. 1.4]{GS} imply that any immersed rational curve in $R_d$ belongs to $R_d^1$. 
\par
For a $C\in R_d$ we have $\pi_\ast\sO_C\simeq \sO_{\oP^1}\oplus \sO_{\oP^1}(-1)^{\oplus (d-1)}$ (since $\deg \pi_\ast\sO_C=-(d+g-1)$, $g=0$, the degrees of the summands are nonpositive, and $h^0(\sO_C)=1$) and thus curves  in $R_d$ correspond to sections of $p:Z_d\to \oP^1$ which arise as projections of sections of $K(E)\to\oP^1$ with $E=\sO_{\oP^1}\oplus \sO_{\oP^1}(-1)^{\oplus (d-1)}$ (and the map associating a section to a curve is a biholomorphism between $R_d$ and the corresponding open subset of the Kodaira moduli space of sections of $p$, cf. Proposition \ref{bihomo}). Equivalently, owing to Proposition \ref{Z(E)}, curves in $R_d$ can be obtained as sections of $Z(E^\prime)$ with $E^\prime\simeq  \sO_{\oP^1}(-1)^{\oplus (d-1)}$.  
\par
If $C\in R_d^2$, then the normal bundle of the corresponding section of $Z_d$ is isomorphic to $\sO_{\oP^1}(1)^{\oplus 2d}$ and, hence, $R_d^2$ comes equipped with a  $\cx$-hyperk\"ahler structure. Moreover, given a real structure $\sigma$ on $|\sO_{\oP^1}(-1)\oplus\sO_{\oP^1}(-1)|$, covering a real structure on $\oP^1$, we obtain a pseudo-hyperk\"ahler or hypersymplectic structure on $(R_d^2)^\sigma$ (depending on whether $\sigma$ is fixed-point free or not).
\begin{example} Let $C$ be a twisted rational curve of degree $3$ not meeting the $\oP^1$. Its ideal is generated by the minors of the $3\times 2$ matrix
$$\begin{pmatrix} x & 0\\ y& x\\ 0 & y\end{pmatrix} -C(t),$$
where the entries of $C$ are constant or linear in $t$.
Thus
$$ x^2= (c_{11}+c_{22})x-c_{12}y-C_3,\enskip xy=c_{32}x+c_{11}y -C_2,\enskip y^2=(c_{21}+c_{32})y-c_{31}x-C_1,
$$
where $C_i$ denotes the determinant of a $2\times 2$ matrix obtained by deleting the $i$-th row from $C$. Thus the matrices $A$ and $B$ are
\begin{equation*} A(t)=\begin{pmatrix} 0 & -C_3 & -C_2\\ 1 & c_{11}+c_{22} & c_{32}\\ 0 & -c_{12} & c_{11}\end{pmatrix},\quad
B(t)=\begin{pmatrix} 0 & -C_2 & -C_1 \\ 0 & c_{32} & -c_{31}\\ 1 & c_{11} & c_{21}+c_{32}
\end{pmatrix}.
\end{equation*}
The metric is given as the coefficient of $t$ in power series expansion of 
\begin{multline*}\tr dA\wedge dB= d(c_{11}+c_{22})\wedge dc_{32}+dc_{32}\wedge dc_{11}+dc_{12}\wedge dc_{31}+da_{11}\wedge d(c_{21}+c_{32})=\\
 \hspace{1.975cm} = d(c_{11}+c_{22})\wedge d(c_{32}+c_{21})-dc_{22}\wedge dc_{21}+dc_{12}\wedge dc_{31}.\hfill\end{multline*}
 The antiholomorphic involution $\sigma$, covering the antipodal map on $\oP^1$,
 acts on linear polynomials $c_{ij}(t)$ as
$$\begin{pmatrix} c_{11}(t) & c_{12}(t)\\ c_{21}(t) & c_{22}(t)\\ c_{31}(t) & c_{32}(t)\end{pmatrix}\mapsto t\ol{\begin{pmatrix} c_{32}(-1/\bar t) & c_{31}(-1/\bar t)\\ -c_{22}(-1/\bar t) & c_{21}(-1/\bar t)\\ 
-c_{12}(-1/\bar t) & -c_{11}(-1/\bar t)\end{pmatrix}}.$$
Restricting the above formula for $\tr dA\wedge dB$ to $\sigma$-invariant sections shows that the hyperk\"ahler metric on the space of real twisted cubics is flat with signature $(8,4)$. 
\end{example}
We shall now describe $R^2_d$, $d\geq 4$, as a  $\cx$-hyperk\"ahler quotient of a flat $\cx$-hyperk\"ahler manifold.  Let $M=\Mat_{d-1,d-1}(\cx)\otimes \cx^4$ and write its elements  as $(X_0,X_1,Y_0,Y_1)$. $M$ is a $\cx$-hyperk\"ahler manifold with the action of $\Mat_{2,2}(\cx)$ on $TM$ given by the left multiplication on $ \Mat_{2,2}(\cx)\simeq \cx^4$, i.e.:
\begin{equation} \begin{pmatrix} X_0 & Y_0\\ X_1 & Y_1\end{pmatrix}\mapsto \begin{pmatrix} aI & bI\\ cI & dI\end{pmatrix} \begin{pmatrix} X_0 & Y_0\\ X_1 & Y_1\end{pmatrix}.\label{Mat}\end{equation}
The symmetric holomorphic $(2,0)$-tensor $g$ is
\begin{equation} g=\tr\bigl( dX_1dY_0-dX_0 dY_1).\label{g}\end{equation}
The twistor space $Z$ of $M$ is the total space of $\Mat_{d-1,d-1}(\cx)\otimes\sO_{\oP^1}(1)^{\oplus 2}$, and an element $(X_0,X_1,Y_0,Y_1)$ of $M$ is identified with the section  $$(X(t),Y(t))=(X_0+tX_1, Y_0+tY_1)$$ of $Z\to \oP^1$. The twisted fibrewise symplectic form on $Z$ is $\tr d(X_0+tX_1)\wedge d(Y_0+tY_1)$.
\par
We define a $\cx$-hyperk\"ahler submanifold $M_d$ of $M$ to consist of $(X_0,X_1,Y_0,Y_1)$ such that:
\begin{itemize}
\item[(i)] $\forall_{t\in\oP^1}$ $(X(t),Y(t),\langle e_1,e_2\rangle)$ is stable.
\item[(ii)] $\forall_{t\in\oP^1}$ $X_{12}(t)=Y_{11}(t),\enskip X_{22}(t)=Y_{21}(t)$.
\end{itemize}
Next we define the relevant group $\hat G_d$ acting on $M_d$. It consists of invertible $d\times d$ matrices of the form
\begin{equation} g(t)=\begin{pmatrix} 1 & 0 & 0 & u_1(t) &\dots &u_{d-3}(t)\\ 0&1&0& \ast &\dots &\ast \\ 0&0&1& \ast  &\dots & \ast \\ 0& 0 &0&\ast &\dots &\ast \\ \vdots &\vdots & \vdots & \vdots& &\vdots \\ 0& 0 &0&\ast&\dots &\ast  \end{pmatrix},\label{Gd}\end{equation}
where asterisks denote complex numbers, and $u_i(t)$, $i=1,\dots, d-3$, are linear polynomials. The group structure is given by matrix multiplication.
Observe that if we  fix $t$, we obtain a group $\hat G_d(t)$ isomorphic to the group $\hat G_0$ of \S\ref{notlinear}. We have $\hat G_d\simeq G_0\ltimes L$, where $G_0$ is obtained by deleting the first row and the first column of \eqref{Gd} and $L\simeq \cx^{2(d-3)}$ is the additive group of $d-3$ linear polynomials. If we write an element of $L$ as $u(t)=(0,0,u_1(t),\dots,u_{d-3}(t))$, then the action of $\hat G_d$ on pairs of linear matrix polynomials $(X(t),Y(t))$ can be written as (cf. \eqref{action}):
\begin{equation} (g_0,u(t)).(X(t),Y(t))=\bigl( g_0X(t)g_0^{-1} -e_1u(t)g_0^{-1}, g_0Y(t)g_0^{-1} -e_2u(t)g_0^{-1}\bigr).\label{action2}
\end{equation}
This action preserves the $\cx$-hyperk\"ahler structure of $M_d$ and it lifts to an action on the twistor space of $M_d$: an element $g(t)$ acts as $g(t_0)$ on the fibre over $t_0$.
We easily compute the $\cx$-hyperk\"ahler moment map $\mu:M_d\to \hat\g_d^\ast\otimes \ssl(2,\cx)\simeq (\g_0^\ast\oplus \fL^\ast) \otimes \ssl(2,\cx)$:
\begin{equation}
\mu_{ij}=\begin{cases}\frac{1}{2} \bigl( \pi([X_0,Y_1]+[X_1,Y_0]),(X_1+tX_0)e_2-(Y_1+tY_0)e_1\bigr) & \text{if $ij=11$},\\
\bigl( \pi([X_0,Y_0]),X_0e_2-Y_0e_1\bigr) &  \text{if $ij=21$},\\
 \bigl( \pi([X_1,Y_1]),t(X_1e_2-Y_1e_1)\bigr)&  \text{if $ij=12$},\end{cases} \label{moment}\end{equation}
where $\pi:\Mat_{d-1,d-1}\to \g_0^\ast$ is the projection onto the last $d-3$ rows.
 We conclude:
\begin{lemma} An element $(X_0,X_1,Y_0,Y_1)$ of $ M_d$ belongs to $\mu^{-1}(0)$ if and only if, for every $t\in \oP^1$, $X(t)e_2=Y(t)e_1$ and  $[X(t),Y(t)]$ has nonzero entries only in the first two rows.\hfill $\Box$\label{mu(0)}\end{lemma}
\begin{remark} We should like to remark that the moment map equations for $\hat G_d$ are formally similar to the {\em complex ADHM equations} considered in \cite{FJ}. The difference is that the $\bigl((d-1)\times 2, 2\times (d-1)\bigr)$-part $(i,j)$ (cf. \S \ref{notlinear}) of the ADHM-datum is no longer linear in $t$; $i$ is now constant, while $j$ is quadratic.
\end{remark}
\begin{lemma} The action of $\hat G_d$ on $\mu^{-1}(0)$ is free and proper.\end{lemma}
\begin{proof}  If $g\in \hat G_d $ fixes a point  $(X_0,X_1,Y_0,Y_1)$ of $ M_d$, then $g(t_0)$ fixes $(X(t_0),Y(t_0))$, for every $t_0$. Choosing a $t_0$ such that $g(t_0)\neq 1$ contradicts Theorem \ref{nonlin}. Therefore   $\hat G_d$ acts freely. The moment map for the twisted symplectic form $\tr dX(t)\wedge dY(t)$ is $\mu(t)=\mu_{21}+2\mu_{11}t+\mu_{12}t^2$. Setting $X_0=\mu(0)^{-1}(0),\; X_1=\mu(1)^{-1}(0),\; X_2=\mu(\infty)^{-1}(0)$, we have $\mu^{-1}(0)=X_0\cap X_1\cap X_2$. For every $t_0$, the group $\hat G_{d}(t_0)$ (obtained by evaluating \eqref{Gd} at $t_0$) acts properly on $\mu(t_0)^{-1}(0)$, owing to Theorem  \ref{nonlin}. It follows that, if $m_i\in \mu^{-1}(0)$, $g_i=(g_i^0,u_i(t))\in \hat G_d$, $i\in \oN$, and $m_i\to m$, $g_i m_i\to m^\prime$, then both $(g_i^0, u_i(0))$ and $(g_i^0,u_i(\infty))$ have convergent subsequences. Therefore $(g_i)$ has a convergent subsequence and $\hat G_d$ acts properly on $\mu^{-1}(0)$. 
\end{proof}
This lemma implies that $\mu^{-1}(0)/\hat G_d$ is a Hausdorff topological space, and that the smooth part of $\mu^{-1}(0)$ is a principal $\hat G_d$-bundle over a complex manifold. Observe now that the fundamental vector field corresponding to $u_0+tu_1\in \fL$ is $$X_u=(-e_1u_0,-e_1u_1,-e_2u_0,-e_2u_1)$$
and, consequently, the subspace of these vector fields is $\Mat_{2,2}(\cx)$-invariant, owing to \eqref{Mat}. Observe also that $g(X_u,X_\rho)=0$ at points of $\mu^{-1}(0)$, for every $\rho\in \hat\g_d$. 
Corollary \ref{smoothQ} implies:
\begin{proposition} Let   $m=(X_0,X_1,Y_0,Y_1)\in \mu^{-1}(0)$ and let $\bar m$ denote its image in $\mu^{-1}(0)/\hat G_d$.
\begin{itemize}\item[(i)] $\bar m$ is a smooth point of $\mu^{-1}(0)/\hat G_d$ if and only if the dimension of the subspace  of $T_mM_d$ spanned by $ AX_\rho$, $A\in\ssl(2,\cx),\;\rho \in \g_0$, equals $3\dim \g_0$;
\item[(ii)] $\bar m$ is a smooth point of $\mu^{-1}(0)/\hat G_d$ and the $\cx$-hyperk\"ahler structure is nondegenerate at $\bar m$ if and only if the quadratic form \eqref{g} is nondegenerate on the subspace generated by $X_\rho$, $\rho \in \g_0$.\hfill $\Box$\end{itemize}
\end{proposition}
Let us denote by $M_d^0$ (resp. $M_d^2$) the subset of $M_d$ where the condition (i) (resp. (ii)) is satisfied. Thus $\mu^{-1}(0)\cap M_d^0/\hat G_d$ is a complex manifold and $\mu^{-1}(0)\cap M_d^2/\hat G_d$ is a $\cx$-hyperk\"ahler manifold. We can now state and prove our main result:
\begin{theorem} There is a natural bijection between $R_d$ and $\mu^{-1}(0)/\hat G_d$. This bijection is a biholomorphism between $R_d^0$ and 
 $\bigl(\mu^{-1}(0)\cap M_d^0\bigr)/\hat G_d$ and a $\cx$-hyperk\"ahler isomorphism between $R_d^2$ and 
 $\bigl(\mu^{-1}(0)\cap M_d^2\bigr)/\hat G_d$.\label{main}\end{theorem}
 \begin{proof} 
 Propositions \ref{bihom} and \ref{Z(E)} show that there is a natural isomorphism between $R_d$ and sections of $Z(E^\prime)$ with $E^\prime \simeq \sO_{\oP^1}(-1)^{\oplus (d-1)}$. From the construction, $Z(E^\prime)$ is the twistor space of $\mu^{-1}(0)/\hat G_d$. Sections of $Z(E^\prime)$ are projections of sections of $K(E^\prime)$, which in turn correspond to points of $\mu^{-1}(0)$ (as shown in  Lemma \ref{mu(0)}). The smooth locus of the (component of) Hilbert scheme of sections of $Z(E^\prime)$ is then the smooth locus of $\mu^{-1}(0)/\hat G_d$, i.e.\ $\mu^{-1}(0)\cap M_d^0/\hat G_d$, and the locus of sections, the normal bundle of which splits as $\sO_{\oP^1}(1)^{\oplus 2d}$, is $\bigl(\mu^{-1}(0)\cap M_d^2\bigr)/\hat G_d$. The result follows.
 \end{proof}
 \begin{remark}  The bijection in the first statement of this theorem is an isomorphism if we give $\mu^{-1}(0)/\hat G_d$ the scheme structure of the Hilbert scheme of sections of $Z(E^\prime)$. A natural question is whether this scheme structure can be obtained via nonreductive invariant theory \cite{BDHK}.\end{remark}
 \begin{remark} One can compute the restriction of \eqref{g} to the subspace generated by fundamental vector fields, and thus characterise $M_d^2$ as consisting of points $m$ such that certain linear operator $L_m:\g_0\to \g_0^\ast$ is invertible. 
 \end{remark}
 
 \subsection{Real structures}
 Up to the action of $PGL(4,\cx)$, $\oP^3$ has two antilinear involutions:
 $$ \sigma[z_0,z_1,z_2,z_3]= [-\bar z_1,\bar z_0,-\bar z_3,\bar z_2],$$
 $$ \sigma^\prime [z_0,z_1,z_2,z_3]= [\bar z_1,\bar z_0,\bar z_3,\bar z_2].$$
 Both of them preserve $\oP^1=\{[z_0,z_1,0,0]\}$, and the space of invariant lines in $\oP^3\backslash \oP^1$ is diffeomorphic to $\oR^4$. In the case of $\sigma$, the $\cx$-hyperk\"ahler structure on the space of lines restricts to a flat hyperk\"ahler structure on $\oR^4$, while in the case of $\sigma^\prime$, it restricts to a flat hypersymplectic structure \cite{DS}. The involutions $\sigma,\sigma^\prime$ act on curves in $\oP^3\backslash \oP^1$ and we obtain a pseudo-hyperk\"ahler (resp.  hypersymplectic) structure on manifolds of $\sigma$-invariant (resp.  $\sigma^\prime$-invariant) cohomologically stable curves of fixed genus and degree. We want to describe these manifolds in the case of genus $0$ curves, i.e. $\bigl(R_d^2)^\sigma$ and  $\bigl(R_d^2)^{\sigma^\prime}$. First of all we have:
 \begin{proposition} $R_d^\sigma$ is empty if $d$ is even.  
 \end{proposition}
 \begin{proof} For $C$ in $R_d$ and the corresponding $\pi:C\to \oP^1$ we have, as observed before, $\pi_\ast \sO_C\simeq \sO_{\oP^1}\oplus\sO_{\oP^1}(-1)^{\oplus (d-1)}$.
 The involution $\sigma$ induces an antilinear involution on $W=\sO_{\oP^1}(-1)^{\oplus (d-1)}$, which covers the antipodal map on $\oP^1$. This, in turn, induces an antilinear involution on $\Lambda^{d-1}W^\ast\simeq \sO_{\oP^1}(d-1)$, which covers the antipodal map. Since this involution has no fixed points, the number of zeros of any section of $\Lambda^{d-1}W^\ast$ must be even.
 \end{proof}
 If $d$ is odd, then there is an induced involution $\sigma$ on $E^\prime\simeq\sO_{\oP^1}(-1)^{\oplus (d-1)}$. Modulo conjugation we can write $\sigma$ in the standard trivialisation over $t\neq \infty$ as 
 $$\sigma(t;f_1,\dots,f_{d-1})=(-1/\bar t; -\bar t\bar f_2, \bar t\bar f_1, -\bar t\bar f_4, \bar t\bar f_3,\dots, -\bar t\bar f_{d-1}, \bar t\bar f_{d-2}).
$$
This, in turn, yields an antiholomorphic involution on $\End(E^\prime)$: $M\mapsto \sigma M\sigma$ and, finally, on $M_d$:
\begin{equation*} \bigl(X(t),Y(t)\bigr)\mapsto \bigl(-\sigma Y(t)\sigma/\bar t, \sigma X(t)\sigma/\bar t\bigr).
\end{equation*}
It follows that:
\begin{equation} (X_0,X_1,Y_0,Y_1)\mapsto\bigl(-\tau \bar Y_1\tau, \tau \bar Y_0\tau, \tau \bar X_1\tau,-\tau \bar X_0\tau\bigr),
\label{sigma}
\end{equation}
where 
$$\tau_{ij}=\begin{cases} -1 &\text{if $j=i+1$ and $i$ is odd},\\ 1 &\text{if $j=i-1$ and $i$ is even},\\ 0 &\text{otherwise}.\end{cases}$$
Consequently:
$$ M_d^\sigma=\{ (X_0,X_1,Y_0,Y_1)\in M_d;\; Y_0=\tau \bar X_1\tau,\enskip Y_1=-\tau \bar X_0\tau\}.$$
An easy computation shows that the quadratic form \eqref{g} restricted to $M_d^\sigma$ is a pseudo-Riemannian metric of signature $\bigl(2(d-1)^2+2(d-3), 2(d-1)^2-2(d-1)\bigr)$. The subgroup $\hat G_d^\sigma$ commuting with $\sigma$ consists of elements $(h,u)\in G_0\ltimes L$, where
$h=\tau\bar h\tau^{-1}$ and $u_{2i-1}(-1/\bar t)=-\ol{u_{2i}(t)}/\bar t$, $i=1,\dots,(d-3)/2$. We conclude from Theorem \ref{main}:
\begin{theorem} Let $d$ be odd. The pseudo-hyperk\"ahler manifold $\bigl(R^2_d\bigr)^\sigma$ of cohomologically stable connected $\sigma$-invariant Cohen-Macaulay curves of arithmetic genus $0$ in $\oP^3\backslash \oP^1$ is isomorphic to the hyperk\"ahler quotient of $\bigl(M_d^2\bigr)^\sigma$ by the group $\hat G_d^\sigma$.\hfill $\Box$ \label{sigmamain}\end{theorem}
\begin{remark} It follows from the main result of \cite{BZ}  that the signature of the pseudo-hyperk\"ahler metric on $\bigl(R^2_d\bigr)^\sigma$ is $(2d+2,2d-2)$.
\end{remark}
\par
There is an analogous description of the hypersymplectic manifold  $\bigl(R^2_d\bigr)^{\sigma^\prime}$. Essentially, we just have to remove all the minus signs in the above formulae (in particular, $d$ can be arbitrary). 
We leave the details to the interested reader.

\end{document}